\documentclass{amsart}

\usepackage{amscd}
\usepackage{amsfonts}
\usepackage{amsmath}
\usepackage{amssymb}
\usepackage{amsthm}
\usepackage[english]{babel}
\usepackage{epic}
\usepackage[mathscr]{eucal}
\usepackage{extpfeil}
\usepackage{faktor}
\usepackage[T1]{fontenc}
\usepackage{graphics}
\usepackage{graphicx}
\usepackage{indentfirst}
\usepackage[utf8]{inputenc}
\usepackage{mathabx}
\usepackage{mathrsfs}
\usepackage{pict2e}
\usepackage{setspace}
\usepackage{t1enc}
\usepackage{tikz}
\usepackage{tikz-cd}
\usepackage{xfrac}

\newtheorem{theo}{Theorem}[section]
\newtheorem{theoalph}{Theorem}
\newtheorem{cor}[theo]{Corollary}
\newtheorem{lem}[theo]{Lemma}
\newtheorem{prop}[theo]{Proposition}
\newtheorem*{claim}{Claim}

\theoremstyle{definition}

\newtheorem{defin}[theo]{Definition}

\newtheorem{quest}[theo]{Question}
\newtheorem*{prob}{Problem}

\theoremstyle{remark}

\newcommand{\bC}{\mathbf{C}}
\newcommand{\bN}{\mathbf{N}}
\newcommand{\bP}{\mathbf{P}}
\newcommand{\bR}{\mathbf{R}}
\newcommand{\bS}{\mathbf{S}}
\newcommand{\bt}{\mathbf{t}}
\newcommand{\bT}{\mathbf{T}}
\newcommand{\kb}{\mathfrak{b}}
\newcommand{\kc}{2^{\aleph_0}}
\newcommand{\kg}{\mathfrak{g}}
\newcommand{\sA}{\mathscr{A}}
\newcommand{\sB}{\mathscr{B}}

\newcommand{\sD}{\mathscr{D}}
\newcommand{\sF}{\mathscr{F}}
\newcommand{\sN}{\mathscr{N}}

\newcommand{\B}{\mathsf{B}}
\newcommand{\ch}{\mathrm{CH}}
\newcommand{\zfc}{\mathrm{ZFC}}

\newcommand{\sone}[2]{\mathsf{S}_1(#1,#2)}

\DeclareMathOperator{\diam}{\mathrm{diam}}
\DeclareMathOperator{\dom}{\mathrm{dom}}

\tikzset{
  symbol/.style={
    draw=none,
    every to/.append style={
      edge node={node [sloped, allow upside down, auto=false]{$#1$}}}
  }
}

\begin{document}

\title{Borel's conjecture and meager-additive sets}

\author[Daniel Calder\'on]{Daniel Calder\'on}
\address{Department of Mathematics, University of Toronto, 40 St George Street, Toronto, Ontario, M5S 2E4, Canada}
\email{d.calderon@mail.utoronto.ca}

\subjclass[2020]{Primary: 03E35, 03E75, 03E15. Secondary: 54F65.}

\keywords{Borel's conjecture, forcing, meager-additivity, selection principles.}

\begin{abstract}
We prove that it is relatively consistent with $\zfc$ that every strong measure zero subset of the real line is meager-additive while there are uncountable strong measure zero sets (i.e., Borel's conjecture fails). This answers a long-standing question due to Bartoszy\'nski and Judah.
\end{abstract}

\maketitle

\section{Introduction}

In this paper, we continue the study of the structure of strong measure zero sets.\footnote{The reader may consult Section 2 for the definitions of the concepts used in the introduction.} Strong measure zero sets were introduced by Borel in \cite{borel}, and have been studied from the beginning of the previous century. Borel conjectured that every strong measure zero set of real numbers must be countable. A few years later, Sierpi\'nski proved in \cite{sierp28} that if the continuum hypothesis ($\ch$) is assumed, then there exists an uncountable strong measure zero set of reals. Nevertheless, the question about the relative consistency of Borel's conjecture remained open until 1976 when Laver, in his ground-breaking \cite{laver76}, constructed a model of set theory in which every strong measure zero set of reals is countable. In his construction, Laver used Cohen's forcing technique.

A result of Galvin, Mycielski, and Solovay (see \cite{galmicsol79}) provides a characterization of Borel's strong nullity in terms of an algebraic (or translation-like) property for subsets of the real line. By means of this characterization, a strengthening of strong nullity, meager-additivity, appeared on the scene. Meager-additivity, as well as other smallness notions on the real line have received considerable attention in recent years. A 1993 question due to Bartoszy\'nski and Judah (see \cite{barjudah}, or \cite[Problem 12.4]{tsabansurvey}) asks whether strong nullity and meager-additivity have a very rigid relationship, in the following sense:

\begin{prob}[Bartoszy\'nski--Judah, 1993]
Suppose that every strong measure zero set of reals is meager-additive. Does Borel's conjecture follow?
\end{prob}

The main result of this paper is a negative answer to this question.

\begin{theoalph}\label{mainintro}
It is relatively consistent with $\zfc$ that every strong measure zero set of reals is meager-additive, yet Borel's conjecture fails.
\end{theoalph}

For the proof of Theorem \ref{mainintro}, we use the technique of iterated forcing with countable support to construct a model of set theory in which there are uncountable strong measure zero sets, and every strong measure zero set in the final extension appears in some intermediate stage of the iteration. This allows us to ``catch the tail'' in a way such that every strong measure zero set of reals in the final extension is forced to satisfy a certain selection principle in the sense of Scheepers (see \cite{schcomb1}) that implies meager-additivity.

The work is organized as follows: In Section 2 we will introduce (in a way as self-contained as possible) the essential preliminaries to this paper, as well as the terminology that will be used hereby. In Section 3 we will introduce and analyze a forcing notion that will be crucial for our construction. In Section 4 we will offer a proof of Theorem \ref{mainintro}. Finally, in Section 5, we will discuss some concluding remarks and open problems.

\subsubsection*{Acknowledgments}

My sincerest gratitude goes to Stevo Todor\v{c}evi\'c, for his patience, and for the crucial suggestions he gave me while working on this problem and on the earlier drafts of this paper. I also wish to thank Osvaldo Guzm\'an for the knowledge he shared with me, to Ond\v{r}ej Zindulka for pointing out my attention to \cite{hruzind19}, and to the anonymous referee for useful suggestions that considerably improved the presentation of this article.

\section{Preliminaries}

\subsection{Terminology}\label{nota}

We will denote by $\bN$ the set of non-negative integers. Recall that a subset $a$ of $\bN$ is called an \emph{initial segment} if for every $n\in a$, and $m<n$, we have that $m\in a$. Every initial segment of $\bN$ is either finite or equals $\bN$. If $a$ is an initial segment of $\bN$, and $X$ is a countable set, a function $\pi\colon X\to a$ is called a \emph{partition} of $X$. We will use the notation $[n]_\pi:=\pi^{-1}\{n\}\subseteq X$ for each piece of the partition $\pi$. We will say that $\pi$ is a \emph{partition into finite sets} if each $[n]_\pi$ is finite.

We will also denote by $\bR$ the set of real numbers, and by $\bC$ the Cantor space of two-valued sequences $x\colon\bN\to\{0,1\}$; both endowed with the usual structure that makes them Polish (separable and completely metrizable) groups.

\subsection{Meager-additivity and strong nullity}

In \cite{borel}, Borel introduced the notion of strong nullity for sets of real numbers. Recall that if $Y$ is a metric space with distance function $d$, then the \emph{diameter} of its subset $X$ is defined as
\begin{equation*}
    \diam X:=\sup_{x,y\in X}d(x,y).
\end{equation*}
A metric space $X$ has \emph{strong measure zero} if for every sequence $(\varepsilon_n:n\in\bN)$ of positive real numbers, there is an open cover $\{U_n:n\in\bN\}$ of $X$ with $\diam U_n\leq\varepsilon_n$ for every $n\in\bN$. If $Y$ is a metric space, we will denote by $\sN^+(Y)$ the class of subsets of $Y$ that are strong measure zero spaces.

A classical result of Galvin, Mycielski, and Solovay (see \cite{galmicsol79}) establishes a link between Borel's strong nullity and a translation-like property for subsets of the reals. The Galvin--Mycielski--Solovay theorem asserts that a set $X$ of real numbers is a strong measure zero space if, and only if, $X+M\neq\bR$ for every meager set $M\subseteq\bR$. The same result holds if subsets of the Cantor space are considered instead of subsets of the real line. The Galvin--Mycielski--Solovay theorem motivates the definition of meager-additivity. A set $X$ of real numbers is \emph{meager-additive} if $X+M$ is meager for every meager set $M\subseteq\bR$. The notion of meager-additivity can be extended to any topological group. If $Y$ is a topological group, we will denote by $\sN^\star(Y)$ the class of meager-additive subsets of $Y$.

We spell out the following fact, which is essentially \cite[Theorem 3.5]{hruzind19}:

\begin{prop}\label{sigmaideal}
Let $Y$ be a locally compact Polish group. Then both classes $\sN^+(Y)$ and $\sN^\star(Y)$ are $\sigma$-ideals of sets with $\sN^\star(Y)\subseteq\sN^+(Y)$, and if $Y$ is either the real line, or the Cantor space, then $\sN^+(Y)\subseteq\sN(Y)$, where $\sN(Y)$ is the $\sigma$-ideal of Lebesgue-null subsets of $Y$.\qed
\end{prop}

In \cite{zind18}, Zindulka offered a Borel-like characterization of meager-additivity for subsets of the Cantor space $\bC$. A metric space $X$ has \emph{sharp measure zero} if, and only if, for every sequence $(\varepsilon_n:n\in\bN)$ of positive real numbers, there exists an open cover $A=\{U_n:n\in\bN\}$ of $X$ with $\diam U_n\leq\varepsilon_n$, and there is a partition $\pi\colon A\to\bN$ into finite sets such that every $x\in X$ is in all but finitely many elements of the set
\begin{equation*}
    \left\{\bigcup[n]_\pi:n\in\bN\right\}.   
\end{equation*}
Zindulka's results (see \cite[Theorem 1.3]{zind18}) imply that a subset of the Cantor space is meager-additive if, and only if, it has sharp measure zero. By \cite[Theorem 6.4]{zind18}, the same characterization holds for subsets of any Euclidean space $\bR^n$. Moreover, Hru\v{s}\'ak and Zindulka proved in \cite[Theorem 7.6]{hruzind19} that if $Y$ is a locally compact Polish group admitting a two-sided invariant metric, then a subset of $Y$ is meager-additive if, and only if, it has sharp measure zero.

Finally, let $c\colon\bC\to[0,1]$ be the usual map given by
\begin{equation*}
    c(x):=\sum_{n\in\bN}\frac{x(n)}{2^{n+1}}.
\end{equation*}
The following proposition, that is a combination of \cite[Lemma 8.1.12]{barjudbook} and \cite[Proposition 6.2]{zind18}, will be useful later on:

\begin{prop}\label{weisszind}
A subset $X$ of the unit interval $[0,1]$ has strong (sharp) measure zero if, and only if, $c^{-1}(X)$ has strong (sharp) measure zero.\qed
\end{prop}

\subsection{Topological combinatorics}

A \emph{relative open cover} of a subset $X$ of the topological space $Y$ is a family $A$ of open subsets of $Y$ such that $X\subseteq\bigcup A$. Since our definition of a relative cover $A$ depends both on $X$, and its ambient space $Y$, we will often refer to $A$ as a cover of $X|Y$.\footnote{This notation is inspired by the conditional probability of \emph{$X$ given $Y$}.} If either $X$ equals $Y$, or if the ambient space is clear from the context, we will simplify notation as much as possible.

A cover $A$ of $X|Y$ is called:
\begin{itemize}
    \item a \emph{pre-$\gamma$-cover} if every $x\in X$ is in all but finitely many elements of $A$.\footnote{A pre-$\gamma$-cover, in opposition to a $\gamma$-cover (see \cite{schcomb1} for definitions), is not required to be infinite.}
    \item a \emph{$\lambda$-cover} if every $x\in X$ is in infinitely many elements of $A$.
    \item an \emph{$\omega$-cover} if every finite subset of $X$ is included in a single element of $A$, and no element of $A$ covers $X$.
\end{itemize}
We will respectively denote by $P\Gamma[X|Y]$, $\Lambda[X|Y]$, $O[X|Y]$, and $\Omega[X|Y]$ the classes of pre-$\gamma$-covers, $\lambda$-covers, open covers, and $\omega$-covers of $X|Y$.

For the sake of completeness, we include a proof of the following well-known fact:

\begin{lem}[Folklore]\label{omegaone}
For every topological space $Y$, and every $A\in\Omega[Y]$, if $A$ is partitioned into finitely many pieces, then at least one of the pieces is an $\omega$-cover of $Y$. In particular, for every finite subset $F\subseteq A$, we have that $A\setminus F\in\Omega[Y]$.
\end{lem}

\begin{proof}
Let $A\in\Omega[Y]$, and let $\pi\colon A\to\{0,\dots,k-1\}$ be a partition. Clearly, no $[i]_\pi$ has $Y$ as an element. If no $[i]_\pi$, for $i<k$, is an $\omega$-cover of $Y$, choose finite subsets $F_i\subseteq Y$, for $i<k$, such that no element of $[i]_\pi$ includes $F_i$. Then no element of $A$ includes the finite set
\begin{equation*}
    \bigcup_{i<k}F_i\subseteq Y,
\end{equation*}
which is a contradiction. For the second part of the statement, let $U\in A$ be fixed. Since $A=(A\setminus\{U\})\cup\{U\}$, and $U\neq Y$, $A\setminus\{U\}$ is an $\omega$-cover of $Y$. If $F\subseteq A$ is an arbitrary finite set, proceed inductively.
\end{proof}

Let $\sA$ and $\sB$ be classes of relative covers (not necessarily of the same subset) on a space $Y$. We will say that the \emph{selection principle $\sone{\sA}{\sB}$ holds} if, and only if, for every sequence $(A_n:n\in\bN)$ of elements of $\sA$ there exists, for each $n\in\bN$, some open set $U_n\in A_n$ such that the set $\{U_n:n\in\bN\}$ is an element of $\sB$.

Following Ko\v{c}inac and Scheepers in \cite{kocsch03}, a cover $A$ of $X|Y$ is called \emph{$\lambda$-groupable} if it is infinite and there exists a partition $\pi\colon A\to\bN$ into finite sets such that every element of $X$ is in all but finitely many elements of the set
\begin{equation*}
    \left\{\bigcup[n]_\pi:n\in\bN\right\}.
\end{equation*}
We will denote by $G\Lambda[X|Y]$ the class of $\lambda$-groupable covers of $X|Y$.

\subsection{Forcing}

A \emph{forcing notion} is a partially ordered set $\bP$. The elements of $\bP$ are also called \emph{conditions}, and if $p\leq q$ then $p$ is said to \emph{extend} $q$. Two conditions $p$ and $q$ are \emph{compatible} if a single condition extends both of them. A subset $D$ of $\bP$ is called \emph{open} if it contains all extensions of all of its elements. A subset $D$ of $\bP$ is called \emph{dense} if it contains some extension of every condition in $\bP$. A subset $G$ of $\bP$ is a \emph{filter} if it satisfies the following two conditions:
\begin{enumerate}
    \item If $p\in G$, and $p\leq q$, then $q\in G$.
    \item Every two elements of $G$ have a common extension in $G$.
\end{enumerate}

If $\sD$ is a family of dense open subsets of $\bP$, then a filter $G$ is called \emph{$\sD$-generic} if it intersects every element of $\sD$ non-trivially. If a filter $G$ on $\bP$ intersects all dense open subsets of $\bP$ that belong to the transitive model $V$, then $G$ is said to be \emph{$V$-generic}. In this situation, one can define the forcing (or generic) extension $V[G]$ which is a transitive model of set theory that includes $V$ and contains $G$ as an element. The model $V$ is usually referred to as the \emph{ground model}. Our notation is standard and follows \cite{halbsetbook}, \cite{jechbookset}, and \cite{kunbook}; which are also standard references for the general theory concerning the forcing technique.

\section{The forcing notions $\bP(\bt)$}

In this section, we will introduce (and analyze some of the properties of) a variation of Silver's forcing of partial functions into $\{0,1\}$, with their domain included in $\bN$, and such that the complement of each of their domains is infinite (see \cite[Chapter 22]{halbsetbook}). Before introducing this family of forcing notions, we will need some definitions and results from the general theory.

\begin{defin}\label{typeBdefin}
The forcing notion $\bP$ is said to \emph{satisfy Axiom $\B$} if there exists a sequence $(\leq_n:n\in\bN)$ of partial orders on $\bP$ such that:
\begin{enumerate}
    \item\label{1.typeBdefin} For every $n\in\bN$, if $p\leq_n q$, then $p\leq q$.
    \item\label{2.typeBdefin} For every $n\in\bN$, if $p\leq_{n+1}q$, then $p\leq_n q$.
    \item\label{3.typeBdefin} For every sequence $(p_n:n\in\bN)$ of elements of $\bP$ such that $p_{n+1}\leq_n p_n$, there exists a condition $p\in\bP$ such that $p\leq_n p_n$ for every $n\in\bN$.
    \item\label{4.typeBdefin} For every $q\in\bP$, and every $n\in\bN$, if $q\Vdash``\tau\in V"$, then there exist a finite set $H\in V$, and $p\leq_n q$, such that $p\Vdash``\tau\in H"$.
\end{enumerate}
\end{defin}

Recall that the forcing notion $\bP$ is said to be \emph{$\omega^\omega$-bounding} if, and only if, for every $V$-generic filter $G$ on $\bP$, and for every function $f\colon\bN\to\bN$ in $V[G]$, there exists a function $g\colon\bN\to\bN$ in $V$ such that $f(n)<g(n)$ for every $n\in\bN$.

\begin{prop}\label{omegabound}
If the forcing notion $\bP$ satisfies Axiom $\B$, then it is proper and $\omega^\omega$-bounding. 
\end{prop}

\begin{proof}
It is clear that if $\bP$ satisfies Axiom $\B$, then it satisfies Baumgartner's Axiom $\mathsf{A}$ (see \cite[\S7]{baumaxiomA} for definitions), and therefore is proper.

To see that $\bP$ is $\omega^\omega$-bounding, let $G$ be a $V$-generic filter on $\bP$, let $f\colon\bN\to\bN$ be a function in $V[G]$, and set $p_0:=\textbf{1}_\bP$. For each $n\in\bN$, let $H_n\in V$ be a finite set, and let $p_{n+1}\leq_n p_n$ be such that $p_{n+1}\Vdash``f(n)\in H_n"$. Let $g\colon\bN\to\bN$ be the function in $V$ defined by $g(n):=\max H_n+1$. If $p\leq_n p_n$ for every $n\in\bN$, then
\begin{equation*}
    p\Vdash``(\forall n\in\bN)(f(n)<g(n))",
\end{equation*}
and therefore, in $V[G]$, $f(n)<g(n)$ for all $n\in\bN$.
\end{proof}

\begin{cor}\label{iteraomegabound}
If $\bP_{\omega_2}$ is a countable support iteration of forcing notions, all of them satisfying Axiom $\B$, then $\bP_{\omega_2}$ is proper and $\omega^\omega$-bounding.
\end{cor}

\begin{proof}
Follows from Theorem III.3.2 and Theorem V.4.3 in \cite{shepropbook}.
\end{proof}

\begin{prop}\label{itershe}
If $\bP_{\omega_2}$ is a countable support iteration of forcing notions, all of them satisfying Axiom $\B$, such that $\bP_{\omega_2}$ has the $\aleph_2$-chain condition, and $\Vdash_\alpha\ch$ for all $\alpha<\omega_2$, then $\Vdash_{\omega_2}``\sN^+(\bC)\subseteq[\bC]^{\leq\aleph_1}"$.
\end{prop}

\begin{proof}
Every forcing notion satisfying Axiom $\B$ is strongly $\omega^\omega$-bounding in the sense of Goldstern, Judah, and Shelah in \cite[Definition 1.13]{goljudshe}. Then the proposition follows from \cite[Corollary 3.6]{goljudshe} and Proposition \ref{weisszind}.
\end{proof}

If $\pi\colon\bN\to\bN$ is a partition into finite sets such that every $[n]_\pi$ is a finite, non-empty interval, and $\max[m]_\pi<\min[n]_\pi$ whenever $m<n$, we will say that $\pi$ is an \emph{interval partition} of $\bN$. We will denote by $\zeta\colon\bN\to\bN$ the \emph{standard interval partition} of $\bN$ determined by
\begin{equation*}
    [n]_\zeta=\left[\frac{n(n+1)}{2},\frac{(n+1)(n+2)}{2}\right).
\end{equation*}

A function $f$ whose domain is some proper subset of $\bN$ will be called a \emph{partial function}. On the other hand, if the function $g$ has $\bN$ as its domain, we will say that $g$ is \emph{total}. If $f$ is a partial function, then we define the \emph{gap-counting function} $\kg_f\colon\bN\to\bN$ by
\begin{equation*}
    \kg_f(n):=\left|[n]_\zeta\setminus\dom f\right|.
\end{equation*}

A total function $g\colon\bN\to\bN$ is said to be \emph{staggered divergent} if it is non-decreasing and divergent, i.e., $\lim_n g(n)=\infty$. If $g\colon\bN\to\bN$ is divergent, we define
\begin{equation*}
    \mu_n(g):=\min\left\{m\in\bN:g(m)\geq n\right\}.
\end{equation*}
An easy (but useful) observation is that if $f$ is a partial function with divergent gap-counting $\kg_f$, then for every $n\in\bN$ we have that $\kg_f(\mu_n(\kg_f))\geq n$.

Let $\bt=(t_n:n\in\bN)$ be some sequence of sets. A partial function
\begin{equation*}
    f\colon\dom f\to\bigcup_{n\in\bN}t_n    
\end{equation*}
is called a \emph{partial $\bt$-selector} if, and only if, $f(n)\in t_n$ for all $n\in\dom f$.

For the duration of this section, fix a sequence $\bt=(t_n:n\in\bN)$ of finite sets.

\begin{defin}
The forcing notion $\bP(\bt)$ is the set of partial $\bt$-selectors $p$ such that the gap-counting function $\kg_p$ is staggered divergent.

We order $\bP(\bt)$ by $p\leq q$ if $p\supseteq q$.
\end{defin}

The aim of this section is to prove that the forcing notion $\bP(\bt)$ satisfies Axiom $\B$. To do this, we will start by defining, for each $n\in\bN$, a binary relation $\leq_n$ on $\bP(\bt)$ given by $p\leq_n q$ if, and only if, $p\leq q$ and for all $i\leq\mu_n(\kg_q)$,
\begin{equation*}
    [i]_\zeta\setminus\dom q=[i]_\zeta\setminus\dom p.    
\end{equation*}
The relation $\leq_n$ may be thought of in the following way: $p\leq_n q$ if $p$ is an extension of $q$, and these two partial functions are exactly the same until (and including) the first interval of the partition $\zeta$ in which the domain of $q$ avoids at least $n$ non-negative integers. Of course, the expression ``exactly the same'' means that these functions even have the same gaps in their domains.

\begin{lem}
For every $n\in\bN$, the relation $\leq_n$ is a partial order on $\bP(\bt)$.
\end{lem}

\begin{proof}
It is enough to prove that the relation $\leq_n$ on $\bP(\bt)$ is transitive. Suppose that $p\leq_n q\leq_n r$. We want to conclude that $p\leq_n r$. Since $q\leq_n r$, $\mu_n(\kg_q)=\mu_n(\kg_r)$. Since also $p\leq_n q$, we have that for every $i\leq\mu_n(\kg_q)=\mu_n(\kg_r)$,
\begin{equation*}
    [i]_\zeta\setminus\dom r=[i]_\zeta\setminus\dom q=[i]_\zeta\setminus\dom p.
\end{equation*}
Therefore, $p\leq_n r$.
\end{proof}

\begin{theo}\label{typeB}
The forcing notion $\bP(\bt)$ satisfies Axiom $\B$.
\end{theo}

\begin{proof}
The only items in Definition \ref{typeBdefin} that require a proof are $\eqref{3.typeBdefin}$ and $\eqref{4.typeBdefin}$:

$\eqref{3.typeBdefin}$ Let $(p_n:n\in\bN)$ be a sequence of conditions such that $p_{n+1}\leq_n p_n$, and let
\begin{equation*}
    p:=\bigcup_{n\in\bN}p_n.    
\end{equation*}
If it turns out that $p\in\bP(\bt)$, then $p\leq_n p_n$ for all $n\in\bN$. Thus, it is enough to prove that the gap-counting function $\kg_p$ is staggered divergent:

$(3.1)$ To see that $\kg_p$ is non-decreasing, let $m\in\bN$ be fixed. Since
\begin{equation*}
    \lim_{n\to\infty}\mu_n(\kg_{p_n})=\infty,
\end{equation*}
we can choose some $n\in\bN$ such that $m\leq\mu_n(\kg_{p_n})$. Note that for every $i\leq\mu_n(\kg_{p_n})$, it holds that $[i]_\zeta\setminus\dom p_n=[i]_\zeta\setminus\dom p$. In particular, we have that if $i\leq\mu_n(\kg_{p_n})$, then $\kg_p(i)=\kg_{p_n}(i)$. This implies that $\kg_p$ is non-decreasing in the interval $[0,m]$. Since this holds for every $m\in\bN$, $\kg_p$ is non-decreasing.

$(3.2)$ Let us check now that $\kg_p$ is divergent. Fix a non-negative integer $n\in\bN$. As before, for every $i\leq\mu_n(\kg_{p_n})$, we have that $\kg_p(i)=\kg_{p_n}(i)$. Thus,
\begin{equation*}
    \kg_p(\mu_n(\kg_{p_n}))=\kg_{p_n}(\mu_n(\kg_{p_n}))\geq n.    
\end{equation*}
Since $n\in\bN$ was arbitrary, and $\kg_p$ is non-decreasing, we have that $\kg_p$ is divergent.

$\eqref{4.typeBdefin}$ Let $q\in\bP(\bt)$ be such that $q\Vdash``\tau\in V"$, fix $n\in\bN$, and set
\begin{equation*}
    F:=\bigcup_{i\leq\mu_n(\kg_q)}[i]_\zeta\setminus\dom q.
\end{equation*}
Let also $\{s_k:k<l\}$ be the finite set of partial $\bt$-selectors with domain $F$.

We will define $p\leq_n q$ recursively: Start by choosing a condition $q_0\leq q$ such that $\dom q\cup F\subseteq\dom q_0$ and $q_0\upharpoonright F=s_0$. Then choose some $p_0\leq q_0$ such that there exists some $x_0\in V$ such that $p_0\Vdash``\tau=x_0"$. Now, if $k\geq 1$, and we already defined $p_i$, for $i<k$, let $q_k\leq q$ be such that $\dom q_k=\dom p_{k-1}$, $q_k\upharpoonright F=s_k$, and $q_k\upharpoonright(\dom q_k\setminus F)=p_{k-1}\upharpoonright(\dom p_{k-1}\setminus F)$. Then choose some $p_k\leq q_k$ such that there exists some $x_k\in V$ such that $p_k\Vdash``\tau=x_k"$. Finally, let $p:=p_{l-1}\upharpoonright(\dom p_{l-1}\setminus F)$. Clearly, $p\leq_n q$. Now, if $H:=\{x_k:k<l\}\in V$, and $r\leq p$ is arbitrary, we can find a further extension $s\leq r$ such that $F\subseteq\dom s$. If $s\upharpoonright F=s_k$, then $s\leq p_k$, and therefore $s\Vdash``\tau=x_k"$. Thus, $p\Vdash``\tau\in H"$, as required. 
\end{proof}

\section{A proof of the main result}\label{BJ}

In this section we will offer a proof of Theorem \ref{mainintro}.

\begin{defin}
Let $\pi\colon\bN\to\bN$ be a partition into finite sets. A subset $X$ of the topological space $Y$ is called \emph{$\pi$-supernull} if for every sequence $(A_n:n\in\bN)$ of $\omega$-covers of $Y$, there exists an infinite collection of open sets $\{U_n:n\in\bN\}$ such that $U_n\in A_n$ for every $n\in\bN$, and every element of $X$ is an element of all but finitely many elements of the set
    \begin{equation*}
        \left\{\bigcup\left\{U_k:k\in[n]_\pi\right\}:n\in\bN\right\}.
    \end{equation*}
\end{defin}

\begin{prop}\label{bacaneria}
If $X$ is a subset of the Cantor space $\bC$ such that the selection principle $\sone{\Omega[\bC]}{G\Lambda[X|\bC]}$ holds, then $X$ has sharp measure zero. In particular, every $\pi$-supernull subset of $\bC$ has sharp measure zero.
\end{prop}

\begin{proof}
Let $(\varepsilon_n:n\in\bN)$ be a sequence of positive real numbers. Without loss of generality, let us assume that $\varepsilon_{n+1}<\varepsilon_n$ for all $n\in\bN$, and that $\varepsilon_n\to 0$ as $n\to\infty$. Let $B_n$ be the collection of all open subsets $U$ of $\bC$ such that
\begin{equation*}
    \varepsilon_{n+1}<\diam U\leq\varepsilon_n.
\end{equation*}
Let also $\rho\colon\bN\to\bN$ be a partition such that every $[n]_\rho$ is infinite, and let $A_n$ be the set of finite unions $U_{m_{n,0}}\cup\dots\cup U_{m_{n,k-1}}$ such that:
\begin{enumerate}
    \item $m_{n,0}<m_{n,1}<\dots<m_{n,k-1}$ are all elements of $[n]_\rho$.
    \item $U_{m_{n,i}}\in B_{m_{n,i}}$ for every $i<k$.
\end{enumerate}

Since the sequence $\varepsilon_n\to 0$ as $n\to\infty$, we may assume that no $A_n$ has $\bC$ as an element, and therefore every $A_n$ is an $\omega$-cover of $\bC$. Using the selection principle $\sone{\Omega[\bC]}{G\Lambda[X|\bC]}$, choose open sets $V_n\in A_n$ such that $\{V_n:n\in\bN\}\in G\Lambda[X|\bC]$. Now, for each $n\in\bN$, let $k(n)\in\bN$ be such that $V_n=U_{m_{n,0}}\cup\dots\cup U_{m_{n,k(n)-1}}$ in a way such that $m_{n,0}<\dots<m_{n,k(n)-1}$ are all elements of $[n]_\rho$, and $U_{m_{n,i}}\in B_{m_{n,i}}$ for every $i<k(n)$. Since all the $U_{m_{n,i}}$'s have different diameters (in particular they are different), the set $\left\{U_{m_{n,i}}:n\in\bN\,\&\,i<k(n)\right\}$ belongs to $G\Lambda[X|\bC]$. The argument finishes by recalling that $\diam U_{m_{n,i}}\leq\varepsilon_n$.
\end{proof}

The classes $\sN^+(\bR)$ and $\sN^\star(\bR)$ may differ in many models of set theory: If the continuum hypothesis holds, for example, then by \cite[Theorem 2.1]{millerreals} there exists a Luzin set, which is a strong measure zero set that is not meager, and therefore is not meager-additive. Nevertheless, the classes $\sN^+(\bR)$ and $\sN^\star(\bR)$ cannot be \emph{extremely different}. More precisely, it cannot be the case that every meager-additive set of reals in countable, while there exists an uncountable strong measure zero set.

Recall that $\kb$ is the minimal cardinality of a set $\sF$ of functions $f\colon\bN\to\bN$ such that for every $g\colon\bN\to\bN$ there exists some $f\in\sF$ such that there are infinitely many $n\in\bN$ with $g(n)<f(n)$.

\begin{cor}\label{BC}
The following are equivalent.
\begin{enumerate}
    \item Every strong measure zero set of reals is countable.
    \item Every meager-additive set of reals is countable.
\end{enumerate}
\end{cor}

\begin{proof}
It is enough to prove that $\neg(1)\Rightarrow\neg(2)$. We proceed by cases:
\begin{enumerate}
    \item[(i)] If $\kb=\aleph_1$, it is a result of Bartoszy\'nski (see \cite[Theorem 2 (1)]{barto03}) that there exists an uncountable meager-additive subset of the real line.
    \item[(ii)] If $\kb>\aleph_1$, let $X$ be a strong measure zero set of reals with $|X|=\aleph_1$. Since $\kb>\aleph_1$, $X$ has the Hurewicz property (see \cite{hur27}). By \cite[Theorem 8]{fremmillsel}, $X$ is a set of reals that has the Hurewicz property and such that the selection principle $\sone{O[X]}{O[X]}$ holds. By \cite[Theorem 14]{nowschweiss}, this implies that the selection principle $\sone{\Omega[\bR]}{G\Lambda[X|\bR]}$ holds. Therefore, by Proposition \ref{bacaneria}, $X$ has sharp measure zero, and is meager-additive.
\end{enumerate}
This concludes the proof.
\end{proof}

Corollary \ref{BC} implies that the classes $\sN^+(\bR)$ and $\sN^\star(\bR)$ are provably in $\zfc$ \emph{not extremely different} in the sense that it cannot be case that $\sN^\star(\bR)$ coincides with the class of countable subsets of $\bR$, while there exists an uncountable strong measure zero set. The Bartoszy\'nski--Judah problem asks whether the classes of meager-additive sets and strong measure zero sets are \emph{intrinsically distinct} in the sense that they can only be equal if they are trivially equal, i.e., if $\sN^+(\bR)$ and $\sN^\star(\bR)$ coincide, this is because $\sN^\star(\bR)=\sN^+(\bR)=[\bR]^{\leq\aleph_0}$. As announced in the introduction, we will give a negative answer to this question by building a model of set theory in which the notions of meager-additivity and strong nullity coincide, yet there exists an uncountable strong measure zero set.

\begin{theo}\label{main}
If $\zfc$ has a model, then it has a model in which:
\begin{enumerate}
    \item\label{1.main} Every strong measure zero subset of $\bC$ is $\zeta$-supernull.
    \item\label{2.main} There are no unbounded reals over $L$.
    \item\label{3.main} $\sN^+(\bC)=[\bC]^{\leq\aleph_1}$.
    \item\label{4.main} $\kc=\aleph_2$.
\end{enumerate}
\end{theo}

It is worth mentioning that the relative consistency of $\sN^+(\bC)=[\bC]^{\leq\aleph_1}$ was first proved by Corazza in \cite{cora89}, and then independently reproved by Goldstern, Judah, and Shelah in \cite{goljudshe}. Before getting lost in the details of Theorem \ref{main}, let us show how it can be used to prove Theorem \ref{mainintro}.

\begin{proof}[Proof of Theorem \ref{mainintro}]
Let $M$ be a model of set theory in which all the items in Theorem \ref{main} hold. By Proposition \ref{bacaneria}, $\sN^+(\bC)=\sN^\star(\bC)$. Now, if $X$ is a strong measure zero set of real numbers that it is not meager-additive, since both $\sN^+(\bR)$ and $\sN^\star(\bR)$ are $\sigma$-ideals (cf. Poposition \ref{sigmaideal}), we may assume that $X\subseteq[0,1]$. By Proposition \ref{weisszind}, $c^{-1}(X)\in\sN^+(\bC)\setminus\sN^\star(\bC)$; contradiction. An analogous argument proves that, in $M$, there are uncountable strong measure zero sets of the reals.
\end{proof}

Let $\bT=(A_n:n\in\bN)$ be some sequence of $\omega$-covers of the Cantor space. We will construct a sequence of finite sets $\hat\bT:=(t_n:n\in\bN)$ such that each $t_n\subseteq A_n$ is a cover of $\bC$, and all the $t_n$'s are pairwise disjoint: Since $\bC$ is compact, let $t_0\subseteq A_0$ be a finite cover of $\bC$. For every $n\geq 1$, if we already defined $t_k$ for all $k<n$, let
\begin{equation*}
    B_n:=A_n\setminus\bigcup_{k<n}t_k.
\end{equation*}
By Lemma \ref{omegaone}, $B_n$ is an $\omega$-cover of $\bC$, so there exists a finite set $t_n\subseteq B_n$ that covers $\bC$. This finishes the construction.

Fix a sequence $\bT$ of $\omega$-covers of $\bC$, and set $\hat\bT$ as described above.

\begin{lem}\label{densityQ}
For every $n\in\bN$, and every $x\in\bC$, the sets
\begin{align*}
    D_n:= & \left\{p\in\bP(\hat\bT):n\in\dom p\right\}\quad\text{and}\\
    E_x:= & \left\{p\in\bP(\hat\bT):\left(\exists N\in\bN\right)\left(\forall n\geq N\right)\left(\exists k\in[n]_\zeta\cap\dom p\right)\left(x\in p(k)\right)\right\}
\end{align*}
are dense open subsets of $\bP(\hat\bT)$.
\end{lem}

\begin{proof}
That all the $D_n$'s and all the $E_x$'s are open is obvious, so we only need to check the density.
\begin{enumerate}
    \item Let $q\in\bP(\hat\bT)$ be a condition, and let $n\in\bN$ be fixed. We are looking for a partial $\hat\bT$-selector $p$ extending $q$, with $n\in\dom p$, and such that the gap-counting $\kg_p$ is still monotone. If $n\in\dom q$, there is nothing to prove. Otherwise, let $m\in\bN$ be such that $n\in[m]_\zeta$, and let $\{m-i:i<l\}$ be the set of all the non-negative integers such that $\kg_q(m)=\kg_q(m-i)$. For each $i<l$, choose $k_i\in[m-i]_\zeta\setminus\dom q$, with $k_0=n$, and choose arbitrary open sets $U_{k_i}\in t_{k_i}$. Setting $p:=q\cup\{(k_i,U_{k_i}):i<l\}$, we obtain that $p\in\bP(\hat\bT)$ is the desired condition.
    \item Let $q\in\bP(\hat\bT)$, and let $x\in\bC$ be fixed. For each $n\geq\mu_1(\kg_q)$, choose some $k_n\in[n]_\zeta\setminus\dom q$. Since each $t_{k_n}$ is an open cover of $\bC$, there must exist open sets $U_{k_n}\in t_{k_n}$ such that $x\in U_{k_n}$ for all $n\in\bN$. If we define
    \begin{equation*}
        p:=q\cup\left\{\left(k_n,U_{k_n}\right):n\geq\mu_1(\kg_q)\right\},    
    \end{equation*}
    then the condition $p\in\bP(\hat\bT)$ is as desired.
\end{enumerate}
This concludes the argument.
\end{proof}

If $G$ is a $V$-generic filter on $\bP(\hat\bT)$, we define
\begin{equation*}
    \phi_G:=\bigcup G.
\end{equation*}

\begin{theo}\label{sumupgen}
If $G$ is a $V$-generic filter on $\bP(\hat\bT)$, then:
\begin{enumerate}
    \item\label{1.sumupgen} $\phi_G$ is a total function, i.e., it has $\bN$ as its domain.
    \item\label{2.sumupgen} For every $n\in\bN$, $\phi_G(n)\in t_n$. In particular, $\phi_G$ is one-to-one.
    \item\label{3.sumupgen} Every element of $\bC$ is in all but finitely many elements of the set
    \begin{equation*}
        \left\{\bigcup\left\{\phi_G(k):k\in[n]_\zeta\right\}:n\in\bN\right\}.   
    \end{equation*}
\end{enumerate}
\end{theo}

\begin{proof}
Fix a $V$-generic filter $G$ on $\bP(\hat\bT)$.
\begin{enumerate}
    \item For each $n\in\bN$, let $p\in G\cap D_n$. Then $n\in\dom p\subseteq\dom\phi_G$.
    \item If $n\in\dom p$, since $p$ is a partial $\hat\bT$-selector, $\phi_G(n)=p(n)\in t_n$.
    \item For each $x\in\bC$, choose a condition $p\in G\cap E_x$. Then $x$ is in all but finitely many elements of the set
    \begin{equation*}
        \left\{\bigcup\left\{p(k):k\in[n]_\zeta\cap\dom p\right\}:n\in\bN\right\},
    \end{equation*}
    and therefore (3) follows readily.
\end{enumerate}
\end{proof}

We are now in shape to finish the proof of the main result.

\begin{proof}[Proof of Theorem \ref{main}]
The model in which we are interested is obtained by doing a countable support iteration of $\bP(\hat\bT_\alpha)$, for $\alpha<\omega_2$, over a model of $V=L$, where at each stage in the iteration, $V[G_\alpha]\models``\bT_\alpha\text{ is a sequence of }\omega\text{-covers of }\bC"$, and where we have dovetailed so as to ensure that for any $\bT$ such that
\begin{equation*}
    V[G_{\omega_2}]\models``\bT\text{ is a sequence of }\omega\text{-covers of }\bC",
\end{equation*}
then for cofinally many $\alpha<\omega_2$ we have that $\bT=\bT_\alpha$. This dovetailing can be done since there are only continuum many sequences of $\omega$-covers of $\bC$, and the intermediate models satisfy the continuum hypothesis.

\begin{claim}
$V[G_{\omega_2}]\models``\text{If }X\in[\bC]^{\leq\aleph_1}\text{, then }X\text{ is a }\zeta\text{-supernull subset of }\bC"$.
\end{claim}

\begin{proof}[Proof of the Claim]
By \cite[Theorem 16.30]{jechbookset}, the iteration poset $\bP_{\omega_2}$ satisfies the $\aleph_2$-chain condition. This implies that if $X$ is a subset of $\bC$ in $V[G_{\omega_2}]$ with $|X|\leq\aleph_1$, then there exists some $\alpha<\omega_2$ such that $X\in V[G_\alpha]$. Now, if $\bT=(A_n:n\in\bN)$ is such that $V[G_{\omega_2}]\models``\bT\text{ is a sequence of }\omega\text{-covers of }\bC"$, let $\beta\geq\alpha$ be such that $\bT=\bT_\beta$. By Theorem \ref{sumupgen}, there is in $V[G_{\beta+1}]$ an infinite cover $\{U_n:n\in\bN\}$ of $X|\bC$ such that $U_n\in A_n$ for every $n\in\bN$, and such that every element of $X$ is in all but finitely many elements of the set
\begin{equation*}
    \left\{\bigcup\left\{U_k:k\in[n]_\zeta\right\}:n\in\bN\right\}.
\end{equation*}
Therefore, $V[G_{\omega_2}]\models``X\text{ is a }\zeta\text{-supernull subset of }\bC"$.
\end{proof}

We may continue now with the proof of Theorem \ref{main}:
\begin{enumerate}
    \item By Proposition \ref{itershe}, every strong measure zero subset $X$ of $\bC$ in $V[G_{\omega_2}]$ has cardinality at most $\aleph_1$, so it is $\zeta$-supernull.
    \item By Corollary \ref{iteraomegabound}, every function $f\colon\bN\to\bN$ in $V[G_{\omega_2}]$ is dominated by some function $g\colon\bN\to\bN$ in the ground model.
    \item Working in $V[G_{\omega_2}]$: By Proposition \ref{itershe}, $\sN^+(\bC)\subseteq[\bC]^{\leq\aleph_1}$. On the other hand, if $X$ is a subset of $\bC$ with $|X|\leq\aleph_1$, then $X$ is $\zeta$-supernull. By Proposition \ref{bacaneria}, $X$ has sharp measure zero, and by Proposition \ref{sigmaideal}, $X$ has strong measure zero as well.
    \item Follows from the usual argument.
\end{enumerate}
\end{proof}

\section{Concluding remarks}

The following question is due to the anonymous referee:

\begin{quest}\label{questreferee}
Is it relatively consistent with $\zfc$ that $\sN^+(\bR)=\sN^\star(\bR)$, but $\sN^+(\bR)$ is not of the form $[\bR]^{\leq\kappa}$ for any cardinal number $\kappa$?
\end{quest}

We will say that a subset $X$ of the real numbers is a \emph{relative pre-$\gamma$-set} if the selection principle $\sone{\Omega[\bR]}{P\Gamma[X|\bR]}$ holds.

\begin{quest}\label{conj}
Suppose that every strong measure zero set of reals is a relative pre-$\gamma$-set. Does Borel's conjecture follow?
\end{quest}

At first sight, Question \ref{conj} could seem like just a random variation of the Bartoszy\'nski--Judah problem. We will spend the remainder of this section to explain why this is not the case. There exists a forcing notion analogous to $\bP(\hat\bT)$ (see \cite[Theorem 5]{millerrelativegamma}), that adds a pre-$\gamma$-cover by selecting one open set in each coordinate of a given sequence of $\omega$-covers. Fix a sequence $\bT$ of $\omega$-covers of $\bC$. A partial $\bT$-selector $f$ is called \emph{initial} if $\dom f$ is a finite initial segment of $\bN$.

\begin{defin}
The forcing notion $\bS(\bT)$ is the set of pairs $p=(C_p,f_p)$ such that:
\begin{enumerate}
    \item $C_p$ is a finite subset of $\bC$
    \item $f_p$ is an initial $\bT$-selector.
\end{enumerate}
We order $\bS(\bT)$ by $p\leq q$ if:
\begin{enumerate}
    \item $C_p\supseteq C_q$.
    \item $f_p\supseteq f_q$.
    \item $(\forall k\in\dom f_p\setminus\dom f_q)(\forall x\in C_q)(x\in f_p(k))$.
\end{enumerate}
\end{defin}

By \cite[Theorem 6]{millerrelativegamma}, the forcing notion $\bS(\bT)$ has the countable chain condition. Clearly, $\bS(\bT)$ forces that for the sequence of $\omega$-covers $\bT$, we can select one open set in each coordinate of $\bT$ to obtain a pre-$\gamma$-cover of $\bC$. Nevertheless, it is easy to see that $\bS(\bT)$ adds a Cohen real to the universe. This is problematic since if Cohen (in particular unbounded) reals are added to the universe, it is not clear at all how to assure that strong measure zero sets in an extension given by a suitable iteration appeared in some intermediate stage. On the other hand, it is not clear how to force the existence of a pre-$\gamma$-cover by selecting one open set in each coordinate of $\bT$ if countable conditions are used instead of finite ones; once a point avoids infinitely many open sets, it avoids them forever. In particular, one may ask the following:

\begin{quest}
Let $\bT=(A_n:n\in\bN)$ be a sequence of $\omega$-covers of the Cantor space. Is it possible to generically add a cover $\{U_n:n\in\bN\}\in P\Gamma[\bC]$ such that $U_n\in A_n$ for every $n\in\bN$ without adding unbounded (or even Cohen) reals?
\end{quest}

\bibliographystyle{plain}
\bibliography{mybib}
\end{document}